\documentclass[a4paper,reqno,12pt]{amsart}

\usepackage[utf8]{inputenc}
\usepackage{microtype}
\usepackage{amsfonts}
\usepackage{amsmath}
\usepackage{graphicx}

\usepackage{amssymb}
\usepackage{mathrsfs}
\usepackage{mathtools}
\usepackage{enumitem}
\usepackage{comment}

\usepackage{comment}
\usepackage{color}
\usepackage{hyperref}
\usepackage{tikz}
\usetikzlibrary{arrows}
\usepackage{tikz-cd}

\usepackage{a4wide}

\usepackage{empheq}
\numberwithin{equation}{section}

\newtheorem{theorem}{Theorem}[section]
\newtheorem{lemma}[theorem]{Lemma}
\newtheorem{proposition}[theorem]{Proposition}
\newtheorem{corollary}[theorem]{Corollary}

\newtheorem{observation}[theorem]{Observation}
\newtheorem*{theorem*}{Theorem}
\newtheorem*{lemma*}{Lemma}

\theoremstyle{definition}
\newtheorem{definition}[theorem]{Definition}
\newtheorem{remark}[theorem]{Remark}

\DeclareMathOperator{\Stab}{\mathrm{Stab}}
\DeclareMathOperator{\rist}{\mathrm{Rist}}

\DeclareMathOperator{\SL}{\mathrm{SL}}

\DeclareMathOperator{\Aut}{\mathrm{Aut}}

\usepackage{ifthen}
\usepackage{xargs}
\newcommand{\optionalarg}[2]{
\ifthenelse{\equal{#2}{}}{%
#1}{%
#1(#2)}
}

\begin{document}

\title[The congruence subgroup problem for a family of branch groups]{The congruence subgroup problem for a family of branch groups}
\date{\today}

\subjclass[2010]{Primary 20E08, 20F65, 20E26; Secondary 20E18}
\keywords{profinite completions, branch groups, self-similar groups, congruence subgroup problem}

\author[Rachel Skipper]{Rachel Skipper}
\address{Department of Mathematics, The Ohio State University, Columbus, Ohio}
\email{skipper.26@osu.edu}
\thanks{The author gratefully acknowledges support from the GIF, grant I-198-304.1-2015 ``Geometric exponents of random walks and intermediate growth groups" and from the Simons Foundation, Grant \#245855 to Marcin Mazur. The author would also like to thank the referee for their careful reading of the paper and many useful comments which helped improve the exposition.}

\begin{abstract}
We construct a family of groups which generalize the Hanoi towers group and study the congruence subgroup problem for the groups in this family. We show that unlike the Hanoi towers group, the groups in this generalization are just infinite and have a trivial rigid kernel. We also put strict bounds on the branch kernel. Additionally, we show that these groups have subgroups of finite index with non-trivial rigid kernel. The only previously known group where this kernel is non-trivial is the Hanoi towers group and so this adds infinitely many new examples. Finally, we show that the topological closures of these groups have Hausdorff dimension arbitrarily close to 1.
\end{abstract}

\maketitle
\thispagestyle{empty}

\section*{Introduction}\label{sec:intro}

Branch groups, and more generally groups acting on rooted trees, have been well-studied in recent years as a result of the exotic properties groups in this class can possess. A primary example of this is the Grigorchuk group which was the first group shown to be amenable but not elementary amenable and also the first group shown to have intermediate growth, answering longstanding open questions. Additionally, the Grigorchuk group, followed shortly thereafter by the Gupta-Sidki $p$-groups, provided explicit and tractable examples of Burnside groups, i.e. finitely generated infinite torsion groups. Branch groups also arise in the classification of just infinite groups which serve as the analogue of simple groups for the class of residually finite groups.

A group $G$ acting on a rooted tree has the congruence subgroup property if each subgroup of finite index contains the pointwise stabilizer of the vertices on some level of the tree, a subgroup naturally arising from the tree structure. This property parallels the classical property of the same name for subgroups of $\SL(n, \mathbb{Z})$ for $n>2$. Since the full automorphism group of the tree, $\Aut(\mathcal{T})$, is itself a profinite group, determining whether or not $G$ has the congruence subgroup property amounts to comparing the profinite completion $\widehat{G}$ to the topological closure $\overline{G}$ as a subgroup of $\Aut(\mathcal{T})$ and determining if the congruence kernel, the kernel of the natural surjection $\widehat{G}\twoheadrightarrow \overline{G}$, is trivial. This kernel serves as the measure of the error in studying the group by looking only at the quotients coming from the level stabilizers as opposed to considering all finite quotients.

For branch groups, there exists another naturally occurring family of finite index subgroups, namely the rigid stabilizers defined in Section \ref{sec:groups}. This additional family of subgroups reduces the congruence subgroup problem to separately determining the branch kernel, the kernel of the map $\widehat{G}\twoheadrightarrow \widetilde{G}$, and the rigid kernel, the kernel of the map $\widetilde{G}\twoheadrightarrow \overline{G}$, where $\widetilde{G}$ is the topological completion of the $G$ with respect to the rigid stabilizers.

Many of the most studied branch groups have been shown to have a trivial congruence kernel including the Fabrykowski-Gupta group and the Gupta-Sidki groups \cite{BGS03}, \cite{Gar16b}, the Grigorchuk group and an infinite family of generalizations of the Fabrykowski-Gupta group \cite{Gri00}, and Multi-GGS groups with non-constant accompanying vector \cite{Per07}, \cite{FGU17}, \cite{GUA}.

Pervova \cite{Per07} constructed the first branch groups without the congruence subgroup property. Nevertheless, the groups in her infinite family, periodic EGS groups with non-symmetric accompanying vector, have a non-trivial branch kernel but a trivial rigid kernel. Likewise, the twisted twin of the Grigorchuk group  was found to have a non-trivial branch kernel but a trivial rigid kernel \cite{BS09}.

Even with the existence of infinite families of groups having either a trivial branch and a trivial rigid kernel or a non-trivial branch kernel but a trivial rigid kernel, only one group appearing in the literature has been shown to have a non-trivial rigid kernel. It is the Hanoi towers group on three pegs \cite{BSZ12}, \cite{Ski19}, which we refer to as $G_3$. In this paper, we study a family of generalizations of the Hanoi towers group, $\{G_n \mid n \geq 3 \}$, and show that unlike the Hanoi towers group, the group $G_n$ has a trivial rigid kernel whenever $n\geq 4$. We compute the rigid and level stabilizers and fully compute the congruence kernel for many $n$, placing strict bounds on the kernel for the remaining $n$. Some of the results are proved in three parts since the structure of $G_n$ partially depends on $n$. Although the higher groups do have trivial rigid kernels, we nevertheless find new examples of branch groups having non-trivial rigid kernels coming from certain finite index subgroups of the $G_n$. This adds infinitely many new examples to the only previously known example of the Hanoi towers group.

We remark that the group $G_4$ was studied briefly in \cite{sie09}, but a subtle overgeneralization in the hypotheses of earlier theorems led to some incorrect conclusions. 

We show the following main theorems.

\begin{theorem*}[\ref{thm:4max}, \ref{thm:oddrist}, \ref{thm:evensubgroups}]
The rigid kernel for $G_n$ is trivial if and only if $n\neq 3$. 
\end{theorem*}

\begin{theorem*}[\ref{thm:branchkernel}]
For $n \neq 3$, the branch kernel, and thus the congruence kernel, for $G_n$ is the inverse limit
\[\varprojlim_{m\geq 1} M_n^m\]
 where $M_n$ is a finite abelian group. When $n\geq 5$ is even, $M_n$ is cyclic of order $(n-1)$ and when $n=4$ or $n\geq 5$ is odd, $M_n$ has exponent bounded between $(n-1)$ and $2(n-1)$.
\end{theorem*}

A main tool in proving these theorems is understanding the abelianization of the rigid stabilizers. This knowledge also allows us to prove: 

\begin{theorem*}[\ref{thm:justinfinite}]
$G_n$ is just infinite if and only if $n\neq 3$.
\end{theorem*}

We show that the triviality of the rigid kernel is not necessarily inherited by finite index subgroups, even if they are maximal. In Section \ref{sec:wpandab}, we put a function $\epsilon$ on $G_n$ which is used in the next theorem.

\begin{theorem*}[\ref{thm:finiteindex}]
For $n\geq 4$, let $d >2$ be such that $d \mid (n-1)$ and let $H_{n,d}$ be the set of elements $g$ of $G_n$ with $\epsilon(g)\equiv 0 \bmod d$. Then $H_{n,d}$ is a subgroup of index $d$ in $G_n$ and is a regular branch group with a non-trivial rigid kernel.
\end{theorem*}

The work leading up to the theorems in Section \ref{sec:comp} makes the Hausdorff dimension for the topological closure of $G_n$ straightforward to compute so we include it for completeness.
   
\begin{theorem*}[\ref{thm:haus}]
For $n\geq 3$, the Hausdorff dimension for $\overline{G_n}$ is
\[
  dim_{\mathrm{H}}(\overline{G_n}) =
  \begin{cases}
  	   \vspace{.2cm}
       1-\frac{\log(48)}{\log(331776)} & \text{if $n=4$} \\
       \vspace{.2cm}
       1-\frac{\log(2)}{\log(n!)} & \text{if $n\geq 5$ is even} \\
       1-\frac{\log(2)}{n\log(n!)} & \text{if $n$ is odd}
  \end{cases}
\]
\end{theorem*}

The paper is organized as follows. In Section \ref{sec:groups}, we make precise the congruence subgroup problem for branch groups, describe the generalization of the Hanoi towers group to the $n$-ary tree, and prove basic properties of the groups. In Section \ref{sec:wpandab}, we outline a solution to the word problem that aids in computing the abelianization. In Section \ref{sec:comp}, we compute the level and rigid stabilizers for $G_n$ and use this to prove the first three main results. In Section \ref{sec:finiteindex}, we study some subgroups of finite index in $G_n$. And finally in Section \ref{sec:haus}, we compute the Hausdorff dimension for $\overline{G_n}$.

\subsection{Notation}\label{subsec:notation}
For two group elements $g$ and $h$ we will write $g^h$ to indicate $h^{-1}gh$ and $[g,h]$ for $g^{-1}h^{-1}gh$. Additionally, for any group $G$ and any subset $S \subseteq G$, $\langle\langle S \rangle \rangle$ will denote the normal closure of $S$ in $G$.

\section{The groups}\label{sec:groups}

For notational purposes, we focus here on groups acting on regular rooted trees. A fuller discussion in the more general case of rooted spherically homogeneous trees can be found in \cite{BSZ12}, \cite{Gar16a}, or \cite{Ski19}.

Let $n\geq 2$ and let $X$ be a set of size $n$ called an alphabet, $X^m$ the set of words of length $m$ in $X$, and $X^*$ the set of all finite words over $X$ including the empty word denoted by $\varnothing$. Then a regular rooted $n$-ary tree, $\mathcal{T}$, is the Cayley graph of the free monoid on the set $X$, see Figure \ref{fig:binarytree}. For a vertex $u\in X^*$, define the length of $u$, denoted $|u|$, to be the length of the word in $X^*$ corresponding to $u$.

\begin{figure}[ht]
\begin{center}
\begin{tikzpicture}[line cap=round,line join=round,>=triangle 45,x=.83cm,y=.88cm]
\draw (5.,7.)-- (3.,5.);
\draw (5.,7.)-- (7.,5.);
\draw (3.,5.)-- (2.,3.);
\draw (3.,5.)-- (4.,3.);
\draw (7.,5.)-- (6.,3.);
\draw (7.,5.)-- (8.,3.);
\draw (4.5,7.6) node[anchor=north west] {\small $\varnothing$};
\draw (2.36,5.66) node[anchor=north west] {\small $0$};
\draw (6.9,5.66) node[anchor=north west] {\small $1$};
\draw (1.0,3.66) node[anchor=north west] {\small $00$};
\draw (3.8,3.66) node[anchor=north west] {\small $01$};
\draw (5.0,3.66) node[anchor=north west] {\small $10$};
\draw (7.8,3.66) node[anchor=north west] {\small $11$};
\draw [thick, dotted] (7,2.9) -- (7,2.5);
\draw [thick, dotted] (3,2.9) -- (3,2.5);
\end{tikzpicture}
\end{center}
\caption{$X=\{0,1\}$, binary tree}
\label{fig:binarytree}
\end{figure}
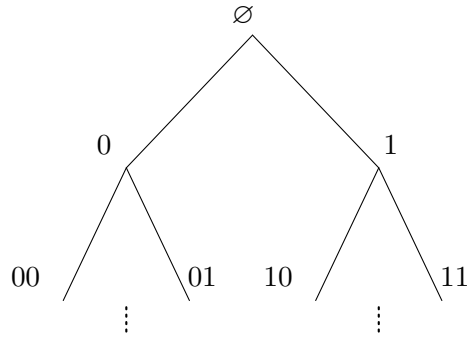

An automorphism of $\mathcal{T}$ is a bijection from $X^*$ to $X^*$ which fixes the root and preserves edge incidences. The symmetric group on $n$ letters, denoted $S_n$, acts in the standard way on $X$ and as such any automorphism $g$ of $\mathcal{T}$ can be described by a labeling of the elements of $X^*$ by permutations 
\[\{g(u) \mid u \in X^*\},\]
where for a vertex $u=x_1x_2 \dots x_m \in X^*$, the action of $g$ on $u$ is described by
\[u^g=x_1^{g(\varnothing)}x_2^{g(x_1)}x_3^{g(x_1x_2)} \cdots x_m^{g(x_1\cdots x_{m-1})}.\]

This gives the isomorphisms 
\[\Aut(\mathcal{T}) \cong \underset{m \geq 1}{\varprojlim}\overset{m \text{ copies }}{(S_n\cdots(\cdots \wr S_n)\wr S_n)\wr S_n}\cong(\cdots(\cdots \wr S_n)\wr S_n)\wr S_n,\]
where on the right-hand side we have the infinitely iterated wreath product of symmetric groups. In particular, this induces the identification
\[\Aut(\mathcal{T})\cong \Aut(\mathcal{T})\wr S_n=\big(\Aut(\mathcal{T})\times \cdots \times \Aut(\mathcal{T})\big)\rtimes S_n\]
where the action of $S_n$ is to permute the coordinates in the product. An element $g\in \Aut(\mathcal{T})$ can be decomposed under this isomorphism as 
\[g=(g_1, \dots, g_n)\sigma\]
 where $\sigma \in S_n$ and $g_i$ is the restriction of the permutation labeling of $g$ to the $i$-th subtree rooted at the first level (canonically identified with the original tree $\mathcal{T}$) and is referred to as the state of $g$ at the $i$-th vertex. Iterating this decomposition, for each $u\in X^*$ we obtain $g_u$, the state of $g$ at the vertex $u$.

\begin{definition}
A group $G\leq \text{Aut}(\mathcal{T})$ is called self-similar if $g_u$ is in $G$ for every $g\in G$ and every $u\in X$. 
\end{definition}

For any subgroup $G$ of Aut$(\mathcal{T})$, four families of subgroups arise naturally from the structure of $\mathcal{T}$.

\begin{definition}
For a vertex $u \in X^*$, the vertex stabilizer, denoted $\Stab_G(u)$, is the set of elements in $G$ which fix the vertex $u$.
\end{definition}

In terms of the labeling of the vertices by elements in a symmetric group, this consists of the elements that necessarily have trivial labeling on all vertices on the geodesic connecting $u$ and $\varnothing$, except possibly at $u$. 

\begin{definition}
For a non-negative integer $m$, the $m$-th level stabilizer, denoted $\Stab_G(m)$, is the normal subgroup $\displaystyle\cap_{|u|=m}\Stab_G(u)$.
\end{definition}

In terms of the labeling, this consists of the elements of $G$ with trivial labeling on all vertices $v$ where $|v|\leq m-1$.  Thus an element $g\in \Stab_G(m)$ will be defined by $|X|^m$ tuple $(g_1, \cdots, g_{n^m})_m$ where each $g_i$ is the state of $g$ at the corresponding vertex on the $m$-th level. Note that $\Stab_G(m)$ has finite index in $G$ for all $m$.

\begin{definition}
For a vertex $u \in X^*$, the rigid stabilizer of the vertex, denoted $\rist_G(u)$, consists of the elements of $G$ which act trivially outside of the subtree rooted at $u$.
\end{definition}

In terms of the labeling, this consists of elements that have trivial labeling on all vertices outside of $\mathcal{T}_u$, the subtree rooted at $u$. If $G$ acts transitively on all the levels of $\mathcal{T}$, then for any two vertices $u$ and $v$ on the same level of the tree, $\rist_G(u)$ is isomorphic to $\rist_G(v)$ (and in fact they are conjugate in $G$). Notationally, for an element $g$ in $\rist_G(u)$, we will write $g=u*\tilde{g}$ where $\tilde{g}=g_u$, the state of $g$ at $u$.  Similarly, for a subgroup $K$ of $\Aut(\mathcal{T})$, we write $v*K=\{v*k \mid k\in K\}$ and $X^m*K=\prod_{|v|=m} v*K$.

\begin{definition}
For a non-negative integer $m$, the $m$-th level rigid stabilizer is the normal subgroup $\rist_G(m)={\langle \rist_G(u)\mid \text{ } |u|=m \rangle} = {\displaystyle\prod_{|u|=m} \rist_G(u)}$, the internal direct product of the rigid stabilizers of the vertices of level $m$. 
\end{definition}

For any group $G$ acting on $\mathcal{T}$, the $m$-th level rigid stabilizer is a subgroup of the $m$-th level stabilizer. Moreover, $\Stab_G(m)$ can be canonically identified with a subgroup of the direct product of $n^m$ copies of $\text{Aut}(\mathcal{T})$ as described above. With this identification, $\rist_G(m)$ is the largest subgroup of $\Stab_G(m)$ which decomposes as a direct product in the same coordinates.

\begin{definition}
A group $G\leq \text{Aut}(\mathcal{T})$ is said to be level transitive if it acts transitively on every level of $\mathcal{T}$.
\end{definition}

Our primary interest here will be in subgroups of $\Aut(\mathcal{T})$ which are branch groups.

\begin{definition}
A group $G\leq \text{Aut}(\mathcal{T})$ is said to be a branch group if it is level transitive and $\rist_G(m)$ has finite index in $G$ for all $m\geq 1$. It is said to be regular branch if it is level transitive and there is a subgroup $K$ with finite index in $G$ such that $v*K\leq K$ for all $v \in X^*$ and such that $X^m*K$ has finite index in $G$ for all $m$. In this case, $K$ is called a branching subgroup.
\end{definition}

If a group is a regular branch group then it is also a branch group as $X^m*K\leq \rist_G(m)$. Note that if $K_1$ and $K_2$ are two branching subgroups for a group $G$ then $\langle K_1, K_2\rangle$ is also a branching subgroup. Thus we define the maximal branching subgroup to be the largest subgroup of $G$ that is branching. Note that the maximal branching subgroup need not be proper. As an example, the Hanoi towers group is a regular branch group with maximal branching subgroup $G_3'$, the derived subgroup of $G_3$ \cite{Ski19}.

The main focus of this paper is on a particular family of groups. For a fixed $n\geq 3$, let $\sigma_i=(1,2, \dots, i-1, i+1, \dots, n-1, n)$, a permutation in $S_n$. Let $a_i$ be the automorphism of the $n$-ary tree defined recursively as follows:
\[a_i=(1,\dots, 1, a_i, 1, \dots, 1)\sigma_i\]
where on the right side of the equation $a_i$ appears in the $i$-th coordinate. 

\begin{definition}
The group $G_n$ is the group generated by $\{a_1, \dots, a_n\}$.
\end{definition}

As mentioned previously, the group $G_3$ is the well-studied Hanoi towers group \cite{GS06}, \cite{GS07}, \cite{BGS03}, \cite{Ski19} whose generators appear in Figure \ref{fig:HTgen}. Our primary focus herein will be on $G_n$, $n\geq 4$. We will recall facts about the Hanoi towers group as they are necessary.

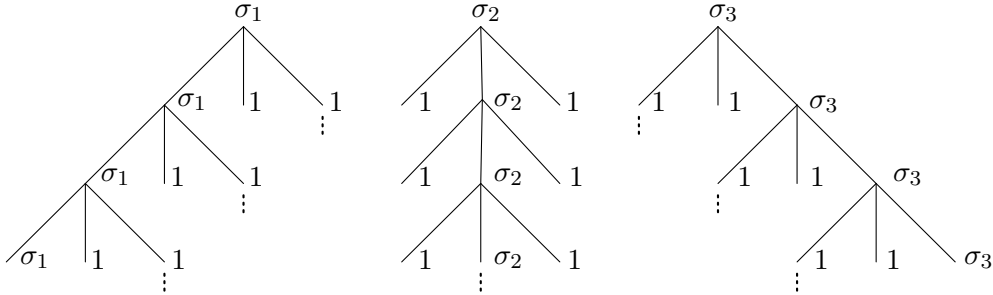
\begin{figure}[ht]
\begin{center}
\begin{tikzpicture}[line cap=round,line join=round,>=triangle 45,x=1.0cm,y=1.0cm, scale=.52];
\draw (0.,0.)-- (-2.,-2.);
\draw (0.,0.)-- (0.,-2.);
\draw (0.,0.)-- (2.,-2.);
\draw (-2.,-2.)-- (-4.,-4.);
\draw (-2.,-2.)-- (-2.,-4.);
\draw (-2.,-2.)-- (0.,-4.);
\draw (-4.,-4.)-- (-6.,-6.);
\draw (-4.,-4.)-- (-4.,-6.);
\draw (-4.,-4.)-- (-2.,-6.);
\draw (6.,0.)-- (4.,-2.);
\draw (6.,0.)-- (6.04,-1.86);
\draw (6.,0.)-- (8.,-2.);
\draw (6.04,-1.86)-- (6.,-4.);
\draw (6.04,-1.86)-- (8.,-4.);
\draw (6.,-4.)-- (4.,-6.);
\draw (6.,-4.)-- (6.,-6.);
\draw (6.,-4.)-- (8.,-6.);
\draw (6.04,-1.86)-- (4.,-4.);
\draw (10.,-2.)-- (12.,0.);
\draw (12.,0.)-- (12.,-2.);
\draw (12.,0.)-- (14.,-2.);
\draw (14.,-2.)-- (12.,-4.);
\draw (14.,-2.)-- (14.,-4.);
\draw (14.,-2.)-- (16.,-4.);
\draw (16.,-4.)-- (14.,-6.);
\draw (16.,-4.)-- (16.,-6.);
\draw (16.,-4.)-- (18.,-6.);
\draw [thick, dotted] (2.,-2.3) -- (2.,-2.8);
\draw [thick, dotted] (0.,-4.3) -- (0.,-4.8);
\draw [thick, dotted] (-2.,-6.3) -- (-2.,-6.8);
\draw [thick, dotted] (6.,-6.3) -- (6.,-6.8);
\draw [thick, dotted] (10.,-2.3) -- (10.,-2.8);
\draw [thick, dotted] (12.,-4.3) -- (12.,-4.8);
\draw [thick, dotted] (14.,-6.3) -- (14.,-6.8);
\begin{small}
\draw[color=black] (0.14,0.36) node {$\sigma_{1}$};
\draw[color=black] (-1.30,-1.92) node {$\sigma_{1}$};
\draw[color=black] (0.3,-1.92) node {$1$};
\draw[color=black] (2.34,-1.92) node {$1$};
\draw[color=black] (-3.24,-3.82) node {$\sigma_{1}$};
\draw[color=black] (-1.66,-3.82) node {$1$};
\draw[color=black] (0.32,-3.82) node {$1$};
\draw[color=black] (-5.28,-5.94) node {$\sigma_{1}$};
\draw[color=black] (-3.68,-5.94) node {$1$};
\draw[color=black] (-1.64,-5.94) node {$1$};
\draw[color=black] (6.14,0.36) node {$\sigma_{2}$};
\draw[color=black] (4.6,-1.92) node {$1$};
\draw[color=black] (6.7,-1.92) node {$\sigma_{2}$};
\draw[color=black] (8.34,-1.92) node {$1$};
\draw[color=black] (4.6,-3.82) node {$1$};
\draw[color=black] (6.7,-3.82) node {$\sigma_{2}$};
\draw[color=black] (8.42,-3.82) node {$1$};
\draw[color=black] (4.6,-5.94) node {$1$};
\draw[color=black] (6.7,-5.94) node {$\sigma_{2}$};
\draw[color=black] (8.38,-5.94) node {$1$};
\draw[color=black] (12.14,0.36) node {$\sigma_{3}$};
\draw[color=black] (10.5,-1.92) node {$1$};
\draw[color=black] (12.44,-1.92) node {$1$};
\draw[color=black] (14.68,-1.96) node {$\sigma_{3}$};
\draw[color=black] (12.68,-3.82) node {$1$};
\draw[color=black] (14.52,-3.82) node {$1$};
\draw[color=black] (16.8,-3.82) node {$\sigma_{3}$};
\draw[color=black] (14.62,-5.94) node {$1$};
\draw[color=black] (16.44,-5.94) node {$1$};
\draw[color=black] (18.6,-5.94) node {$\sigma_{3}$};
\end{small}
\end{tikzpicture}
\end{center}
\caption{The generators $a_1$, $a_2$, and $a_3$ of the Hanoi towers group $G_3$}
\label{fig:HTgen}
\end{figure}

\begin{lemma}
\label{lem:altsym}
For $n\geq 3$, $\langle \sigma_i \mid 1\leq i \leq n \rangle$ is the alternating group on $n$ letters, $A_n$, when $n$ is even and the symmetric group on $n$ letters, $S_n$, when $n$ is odd.
\end{lemma}

\begin{proof}
For all $i$, when $n$ is even $\sigma_i \in A_n$ and when $n$ is odd $\sigma_i \notin A_n$. Further, $\sigma_{i+1}^{-1}\sigma_i=(i, i+1, i+2)$ for $1 \leq i \leq n-2$. Since $\{(i, i+1, i+2)\mid 1\leq i \leq n-2\}$ is a generating set of $A_n$, the result follows. 
\end{proof}

For $g\in \text{Aut}(\mathcal{T})$ and $u\in X^*$, let $\pi_u$ be the projection $g\mapsto g_u$. When the domain of $\pi_u$ is restricted to a subgroup stabilizing the vertex $u$ the map $\pi_u$ is a homomorphism.

\begin{definition}
A self-similar group $G$ is called self-replicating if $\pi_u(\Stab_G(u))=G$ for all $u$.
\end{definition}

If $G$ is both self-replicating and acts transitively on the first level of the tree, then $G$ is level transitive.

\begin{lemma}
\label{lem:allfrac}
For all $n$, $G_n$ is self-replicating.
\end{lemma}

\begin{proof}
If a vertex $v$ is a descendant of $u$ (i.e. $v=uw$ for some $w \in X^*$), then 
\[\pi_v(\Stab_G(v))=\pi_w(\pi_u(\Stab_G(v))).\]
Thus a self-similar group $G$ is self-replicating if and only if $\pi_u(\Stab_G(u))=G$ for every vertex $u$ of level 1.
Suppose $u$ is in the $i$-th coordinate. Then for each $a_j$ and $a_k$ where $k\neq i$, there exists a number $m$ such that $j^{\sigma_k^m}=i$. Moreover, $\sigma_j^{\sigma_k^m}$ fixes $i$. Therefore, $a_j^{a_k^m}$ is in $\Stab_{G_n}(u)$ and $\pi_u(a_j^{a_k^m})=a_j$.
\end{proof}

\begin{corollary}
\label{cor:levtrans}
For all $n$, $G_n$ is level transitive.
\end{corollary}

\section{Word problem and abelianization}\label{sec:wpandab}
We remark that $G_n$ is an example of an automaton group and as such there exists an algorithm in exponential time that solves the word problem \cite{Zuk12}. Here we outline an alternative algorithm for $G_n$ which also allows for the computation of the abelianization. It is a particular case of the algorithm described in Section 3 of \cite{bar03}.

Let $F_n$ be a free group with basis $\{s_1, \dots, s_n\}$. For a freely reduced word $w(s_1, \dots, s_n)=s_{i_1}^{r_1}s_{i_2}^{r_2}\cdots s_{i_k}^{r_k}$, define the length of $w$ to be $|w|=k$. Let $\gamma: F_n \hookrightarrow F_n \wr S_n$ be the map defined by $\gamma(s_i)=(1,\dots, 1, s_i, 1, \dots 1)\sigma_i$ where $s_i$ is in the $i$-th coordinate and $\sigma_i=(1, \dots i-1, i+1, \dots, n)$ as before. In other words, $\gamma$ mimics the recursive definition of $a_i$.

\begin{proposition}
\label{prop:contracting}
Let $w(s_1,\dots, s_n)$ be an element of $F_n$ and suppose $\gamma(w)=(w_1, \dots, w_n)\theta$. Then for all $i$, $|w_i| \leq \frac{|w|+1}{2}$. 
\end{proposition}

\begin{proof}
If $w$ is of length $1$, then $w$ is of the form $s_i^{r}$ so $\gamma(w)=(1, \dots,1, s_i^{r}, 1, \dots, 1)\sigma_i^{r}$ and the claim is true.

Note that $\sigma_{i_1}^{r_1}$ is a permutation of $\{1, \dots, n\} \backslash \{i_1\}$. In particular, if $w=s_{i_1}^{r_1}s_{i_2}^{r_2}$ where $i_1\neq i_2$ then
 $\gamma(w)$ is of the form

\[(1, \dots, s_{i_1}^{r_1}, 1, \dots, 1, s_{i_2}^{r_2}, 1, \dots, 1)\sigma_{i_1}^{r_1}\sigma_{i_2}^{r_2}\]
where $s_{i_2}^{r_2}$ is in the $i_2^{\sigma_{i_1}^{r_1}}$ coordinate and $i_2^{\sigma_{i_1}^{r_1}}\neq i_1$. Again the claim holds.

Now suppose $w=s_{i_1}^{r_1}s_{i_2}^{r_2}\cdots s_{i_k}^{r_k}$ has length $k$ for some $k\geq 3$ and $\gamma(w)=(w_1, \dots, w_n)\theta$. Then for $m=\lceil \frac{k}{2} \rceil$, $w$ can be written as $u_1\cdots u_m$ where $|u_i| \leq 2$ for each $i$. In this case, 
\[\gamma(u_i)=(u_{i_1}, u_{i_2}, \dots, u_{i_n})\theta_i\]
for some $\theta_i \in S_n$ and where for all $i$ between $1$ and $m$  and all $j$ between $1$ and $n$, $|u_{i_j}|$ is either $0$ or $1$. Therefore each $w_i$ is a product of $m$ words of length $0$ or $1$ and $|w_i| \leq \lceil \frac{k}{2} \rceil \leq \frac{k+1}{2}$. \end{proof}

Now let $1 \rightarrow R_n \rightarrow F_n \overset{\phi_0}{\rightarrow} G_n \rightarrow 1$ be a presentation for $G_n$ where $\phi_0(s_i)=a_i$. Since $\gamma$ mimics the recursive definition of the generators of $G_n$, the following diagram commutes:

\[
  \begin{tikzcd}
    F_n \arrow{r}{\gamma} \arrow[swap]{dr}{\phi_0} & Im(\gamma) \arrow{d}{\phi_1} \\
     & G_n
  \end{tikzcd}
\]
where $\phi_1((1, \dots, 1, s_i, 1, \dots 1 )\sigma_i)=a_i$.

This fact along with Proposition \ref{prop:contracting} provide tools for solving the word problem. Indeed, let $w(s_1, \dots, s_n)$ be in $F_n$. If $|w|=1$, then $w(a_1, \dots, a_n)$ is trivial if and only if $w(s_1, \dots, s_n)=s_i^{r(n-1)}$ for some $i$ and $r$. If $|w|\geq 2$, then we can apply $\gamma$ to $w$ to obtain $\gamma(w)=(w_1, \dots, w_n)\theta$ where $|w_j|<|w|$. If $\theta$ is a non-trivial permutation then $w(a_1, \dots, a_n)\neq 1$ and we are done. Similarly, if $\theta$ is trivial and each $w_j$ has length $0$ or $1$, then $w(a_1, \dots, a_n)=1$ if and only if each $w_j(s_1, \dots, s_n)$ is of the form $s_{i_j}^{r_j(n-1)}$. The remaining possibility is that $\theta$ is the trivial permutation and for some $w_j$, the length of $w_j$ is at least 2. In this case repeat the above process to the each $w_j$ until either we find a non-trivial permutation or each obtained word has length at most 1 and is of the form $s_i^{r(n-1)}$.

As a result of the word problem algorithm, the abelianization of $G_n$ is straightforward to compute. First, observe that the generators of $G_n$ have order $(n-1)$ and so $G_n/G_n'$ is a quotient of $(\mathbb{Z}/(n-1)\mathbb{Z})^n$. Now for a word $w(s_1, \dots, s_n)$, let $\epsilon_{s_i}$ be the sum of the exponents on the $s_i$ terms in $w$. Consider now $\gamma(w)=(w_1, \dots, w_n)\theta$. By the way $\gamma$ is defined 
\[\epsilon_{s_i}(w(s_1, \dots, s_n))=\sum_{j=1}^n \epsilon_{s_i}(w_j(s_1, \dots, s_n))\]
The algorithm states that if a word $w(s_1, \dots, s_n)$ produces a trivial word in $G_n$, then after some number of iterations, the sum of the exponents of the $s_i$'s over all the states on a given level is equal to $0$ modulo $n-1$. But this is the same as $\epsilon_{s_i}(w)$. In other words, if $w(a_1, \dots, a_n)=1$ then $\epsilon_{s_i}(w(s_1, \dots, s_n))\equiv 0 \bmod (n-1)$ for all $i$. Thus $R_n\leq \langle F_n', s_1^{n-1}, \dots, s_n^{n-1} \rangle$ and $G_n$ surjects onto $(\mathbb{Z}/(n-1)\mathbb{Z})^n$.

\begin{proposition}
\label{prop:abel}
The abelianization of $G_n$ is $G_n/G_n'\cong(\mathbb{Z}/(n-1)\mathbb{Z})^n$.
\end{proposition}

A similar property to what is described in Proposition \ref{prop:contracting} is frequently studied in the setting of self-similar groups.

\begin{definition} 
A self-similar group $G$ is called contracting if there exists a finite set $\mathrm{N} \subset G$ such that for every $g \in G$, there exists $k\in \mathbb{N}$ such that $g_v \in \mathrm{N}$ for all words $v \in X^*$ of length greater than or equal to $k$. The minimal set $\mathrm{N}$ with this property is called the nucleus of the self-similar action. 
\end{definition}

Since the generators of $G_n$ have order $(n-1)$, the next result follows immediately from Proposition \ref{prop:contracting}.

\begin{corollary}
$G_n$ is contracting with nucleus \[\mathrm{N}=\{1, a_1, \dots, a_1^{n-2}, \dots, a_n, \dots, a_n^{n-2}\}.\]
\end{corollary}

The abelianization also allows us to put some functions on $G_n$ which will be of use to us later.
 
\begin{definition}
Let $g$ be an element of $G_n$. Let $w(s_1, \dots, s_n)=s_{i_1}^{r_1}s_{i_2}^{r_2}\cdots s_{i_k}^{r_k}$ be a word in $s_1, s_2, \dots, s_n$ such that $w(a_1, \dots, a_n)=g$. Define $\epsilon(g)$ to be
\[\epsilon(g)=(\sum_{j=1}^k r_i) \bmod (n-1).\]
\end{definition}

\begin{lemma} 
\label{lem:surject}
$\epsilon: G_n \rightarrow \mathbb{Z}/(n-1)\mathbb{Z}$ is a well defined, surjective homomorphism.
\end{lemma}

\begin{proof}
Since $G_n/G_n' \cong (\mathbb{Z}/(n-1)\mathbb{Z})^n$, $\epsilon$ is the composition of the abelianization map $[Ab]:G_n \rightarrow (\mathbb{Z}/(n-1)\mathbb{Z})^n$ with the map 
$\psi: (\mathbb{Z}/(n-1)\mathbb{Z})^n \rightarrow \mathbb{Z}/(n-1)\mathbb{Z}$ defined by $\psi:(s_1, s_2, \dots s_n) \mapsto \sum_{i=1}^n s_i$. Clearly, this map is well defined and as both $[Ab]$ and $\psi$ are surjective, $\epsilon$ is surjective.
\end{proof}

\begin{definition}
Let $g=(g_1, \dots, g_n)\sigma \in G_n$ where $g_i \in G_n$ for all $i$.  Define
\[\epsilon_1(g)=\sum_{i=1}^n \epsilon(g_i) \bmod (n-1).\]
\end{definition}

\begin{lemma}
\label{lem:epsilonequiv}
For an element $g\in G_n$, $\epsilon(g)=\epsilon_1(g)$.
\end{lemma}
\begin{proof}
This follows from the discussion preceding Proposition \ref{prop:abel}. 
\end{proof}

\section{The congruence subgroup problem}\label{sec:comp}
\begin{definition}
A group $G$ acting on a regular rooted tree has the congruence subgroup property if every subgroup of finite index contains a level stabilizer.
\end{definition}

In the setting of branch groups, this is equivalent to every subgroup of finite index containing a rigid stabilizer and every rigid stabilizer containing a level stabilizer. Since $\Stab_G(m)$ has finite index in $G$ for all $m$ and since this collection forms a descending collection of normal subgroups, taking $\{\Stab_G(m)\mid m \in \mathbb{N} \}$ as a basis for the neighborhoods of $\{1\}$ produces a profinite topology on $G$ (see section 3.1 \cite{RZ10}), called the \textit{congruence topology}. Likewise $\rist_G(m)$ has finite index for all $m$, and in the same way produces a profinite topology called the \textit{branch topology}. Further, $G$ has a third natural topology, the \textit{full profinite topology} where $\mathcal{N}=\{N\unlhd G\mid  |G:N|<\infty\}$ is taken as a basis for the neighborhoods of $\{1\}$. The congruence topology is weaker than the branch topology which is weaker than the full profinite topology. We can complete $G$ in terms of these topologies and obtain three profinite groups:

\begin{align*}
&\overline{G}=\underset{m \geq 1}{\varprojlim} G/\Stab_G(m)  &\textit{the congruence completion}  \\
&\widetilde{G}=\underset{m\geq 1}{\varprojlim} G/\rist_G(m)  &\textit{the branch completion}  \\
&\widehat{G}=\underset{N \in \mathcal{N}}{\varprojlim}G/N  &\textit{the profinite completion}  \\
\end{align*}
Since $\cap_{m\geq 1} \Stab_G(m)=\{1\}$, G is residually finite and embeds into $\overline{G}$, $\widetilde{G}$, and $\widehat{G}$.

Thus $G$ has the congruence subgroup property if and only if $\widehat{G}$ and $\overline{G}$ coincide, that is \textit{congruence kernel}, the kernel of the natural surjection $\widehat{G}\twoheadrightarrow \overline{G}$, is trivial. The \textit{congruence subgroup problem} for branch groups consists not only of determining whether a branch group has the congruence subgroup property but also of quantitatively describing the congruence kernel. Since there is a third topology at play, namely the branch topology, we can instead study two pieces of the congruence kernel, namely the \textit{branch kernel}, the kernel of the natural surjection $\widehat{G}\twoheadrightarrow \widetilde{G}$, and the \textit{rigid kernel}, the kernel of the natural surjection $\widetilde{G} \twoheadrightarrow \overline{G}$. Although a group may have many realizations as a branch group, each of these kernels are invariants of the group and are not dependent on the choice of realization \cite{Gar16a}.

The kernels for $G_3$ are calculated in \cite{BSZ12}.

\begin{theorem*}[\cite{BSZ12}, Theorem 3.11]
The kernel of $\widehat{G_3}\rightarrow \widetilde{G_3}$ is free profinite abelian. The kernel of $\widetilde{G_3}\rightarrow \overline{G_3}$ is a Klein group of order 4. The kernel of $\widehat{G_3}\rightarrow \overline{G_3}$ is metabelian and torsion-free, but is not nilpotent.
\end{theorem*}

A second, more constructive proof for the rigid kernel calculation in the last theorem can be found in \cite{Ski19}.

The first step in computing the kernels for the groups $G_n$, $n\geq 4$, is to understand their rigid stabilizers and level stabilizers.

First we make the following observation.

\begin{observation}
\label{obs:conjugation}
For any vertex $v$, conjugating any element $h\in \rist_{\Aut(\mathcal{T})}(v)$ by an automorphism $g$ of $\mathcal{T}$ works as follows:

Let $|v|=m$ and suppose 
\[h=(1, \dots, 1, h_v, 1, \dots, 1)_m\]
where $h_v$ is in the $v$-th coordinate. Let $g$ decompose as 
\[(g_1, \dots, g_{n^m})\sigma\]
where $\sigma$ is in the $m$-fold iterated wreath product of $S_n$. Suppose $\sigma$ sends the vertex $v$ to the vertex $u$. Then 
\[h^g=(1, \dots, 1, h_v^{g_u}, 1, \dots, 1)_m\]
where $h_v^{g_u}$ is in the $u$-th coordinate.
\end{observation}

For a level transitive, self-replicating group, this significantly reduces the calculations for rigid stabilizers as illustrated by Proposition \ref{prop:reduction}.

\begin{proposition}\label{prop:reduction}
Suppose $G$ is a level transitive, self-replicating group. If $v*g \in G$, then 
\[u*\langle \langle g \rangle \rangle \leq \rist_G(u)\]
for all $u$ with $|u|=|v|$.
\end{proposition}

\begin{proof}
Suppose $v*g \in G$ and that $G$ is a level transitive, self-replicating group. Let $g^{h}$ be a conjugate of $g$ in $G$. Since $G$ is level transitive, for any vertex $u$ on the same level as $v$ there exists $\tilde{h}_1\in G$ such that $\tilde{h}_1$ takes $v$ to $u$. Then by Observation \ref{obs:conjugation}, $(v*g)^{\tilde{h}_1}=u*g^{h_1}$ for some $h_1 \in G$. Since $G$ is self-replicating there exists $\tilde{h}\in \Stab_G(u)$ such that the state of $\tilde{h}$ at $u$ is $h_1^{-1}h$. Then $(v*g)^{\tilde{h}_1 \tilde{h}}=u*g^h$.
\end{proof}

Additionally, for the groups $G_n$ there is another simplification that comes from the symmetry of the generators.

\begin{observation}
\label{obs:simplification}
Let $\omega$ be the permutation $(1,2,\cdots, n)$ and let $\lambda$ be the automorphism of $n$-ary tree defined recursively by 
\[\lambda=(\lambda, \lambda, \dots, \lambda)\omega.\]
Then conjugation by $\lambda$ is an automorphism of the group $G_n$ which takes $a_n \mapsto a_1$ and $a_i \mapsto a_{i+1}$ for $1 \leq i \leq n-1$. Further, if \[g=(g_1, g_2, \dots, g_n)\sigma,\]
then 
\[g^\lambda=(g_n^\lambda, g_1^\lambda, \dots, g_{n-1}^\lambda)\sigma^\omega.\]
\end{observation}

\begin{theorem}
\label{thm:branchingcommutator}
For all $n$, $G_n$ is a regular branch group with branching subgroup $G_n'$.
\end{theorem}

\begin{proof}
By Lemma \ref{lem:allfrac} and Corollary \ref{cor:levtrans}, $G_n$ is level transitive and self-replicating. Therefore, by Proposition \ref{prop:reduction}, it suffices to show that for each $g$ in some normal generating set of $G_n'$ there is some $v\in X$ such that $v*g\in G_n$. And finally, by Observation \ref{obs:simplification}, it suffices to find a conjugate of $v*[a_1,a_i]$ for each $i$ between $1$ and $1 + \lfloor \frac{n}{2} \rfloor$ and for some $v\in X$.

The case when $n=3$ is dealt with in \cite{Ski19}.

When $n=4$, we have the following elements:
\[ [a_3^{-a_1},a_3^{-a_2}](a_2^{-1}a_1)^3=(1,1,[a_1,a_2]^{a_2},1)_1\]
\[ [a_2^{a_1^{-1}}, a_2^{a_3}](a_1a_3)^{-3}=(1,[a_1,a_3]^{-a_3^{-1}},1,1)_1\]

When $n=5$, we have the following elements: 
\[ [(a_1a_4^{-1})^2, (a_2a_4^{-1})^2]=([a_1,a_2],1,1,1,1)_1\]
\[ [(a_3^{-1}a_1)^2,(a_3a_1^{-1})^2]=(1,[a_1,a_3],1,1,1)_1\]

When $n=6$, we have he following elements:
\[ [(a_1a_4^{-1})^2,(a_2a_4^{-1})^2]=([a_1,a_2],1,1,1,1,1)_1\]
\[ [(a_3^{-1}a_1)^2, (a_3a_1^{-1})^2]=(1,[a_1,a_3],1,1,1,1)_1\]
\[ [(a_6^{-1}a_1a_2a_1^{-1})^{a_3}, (a_4a_5^{-1}a_4^{-1}a_3)]=(1,1,1,[a_1,a_4],1,1)_1 \]

For the remaining $n$, fix $i$, $1 \leq i \leq 1 + \lfloor \frac{n}{2} \rfloor$ and let $j=i+2 \geq 4$. Then
\[[(a_1a_j^{-1})^2, ((a_ia_j^{-1})^2)^{a_j^{-(i-2)}}]=([a_1, a_i], 1, \dots, 1)_1.\]

Since $G_n'$ has finite index in $G_n$ and we obtain the result.
\end{proof}

\begin{remark}
Note that $G_n'$ is \textit{not} the maximal branching subgroup for $n\geq 4$. The maximal branching subgroup for $G_n$, which depends on the size of $n$ and whether $n$ is even or odd, will be computed in Theorems \ref{thm:4max}, \ref{thm:oddrist}, and \ref{thm:evensubgroups}.
\end{remark}

\begin{definition}
Let $I_n$ be the collection of elements of the form 
\[(1, \dots, 1, g, 1, \dots, 1, \dots, 1, g^{-1}, 1, \dots, 1)_1\]
where $g$ ranges over all elements of $G_n$ and the coordinates in which $g$ and $g^{-1}$ appear ranges over the set $\{1, \dots, n\}$.
\end{definition}

\begin{proposition}
\label{prop:elementinverse}
When $n\geq 4$, $I_n$ is contained in $G_n'$.
\end{proposition}

\begin{proof}
First, we observe that if $g=(g_1, \dots, g_n)_1$ is an element in $\Stab_{G_n}(1)$ and $h=(h_1, \dots, h_n)\sigma$ is an element of $G_n$, then $g^h=(g_{1^\sigma}^{h_{1^\sigma}}, \dots, g_{n^\sigma}^{h_{n^\sigma}})_1$ which is equivalent to $(g_{1^\sigma}, g_{2^\sigma} , \dots, g_{n^\sigma})_1$ modulo $G_n'$ by Theorem \ref{thm:branchingcommutator}.

Consider the element 
\[[a_1^{a_2}, a_3]=(1, a_2^{-1}, [a_1, a_3], a_2^2, a_2^{-1}, 1, \dots, 1)_1\equiv (1, a_2^{-1}, 1, a_2^2, a_2^{-1},1, \dots, 1)_1 \bmod G_n'\] 
where the equivalence is again by Theorem \ref{thm:branchingcommutator}. Letting $\delta=(1, a_2^{-1}, 1, a_2^2, a_2^{-1},1, \dots, 1)_1$, we see that
\[\delta\delta^{-a_1a_3^{-1}}=(1, a_2^{-1}, a_2^{a_3^{-1}},1, \dots, 1)_1\equiv (1, a_2^{-1}, a_2,1, \dots, 1)_1 \bmod G_n'.\]
Since $G_n$ acts as either $A_n$ or $S_n$ on the first level, by our first observation all elements of the form $(1, \dots, 1, a_2, 1,\dots, 1, a_2^{-1}, 1, \dots,1)_1$ with the $a_2$ and $a_2^{-1}$ in any coordinate are contained in $G_n'$. Similarly, by Observation \ref{obs:conjugation} all elements of the form $(1, \dots, 1, a_i, 1,\dots, 1, a_i^{-1}, 1, \dots,1)_1$ for $1 \leq i \leq n$ are likewise in $G_n'$.

Finally suppose $g=a_{i_1}^{m_{i_1}}\cdots a_{i_k}^{m_{i_k}}$. Then,
\[(1, \dots, 1, a_{i_1}, 1, \dots, 1, a_{i_1}^{-1},1, \dots, 1)_1^{m_{i_1}}\cdots (1, \dots, 1, a_{i_k}, 1, \dots, 1, a_{i_k}^{-1},1, \dots, 1)_1^{m_{i_k}}\]
\[=(1, \dots, 1, g, 1, \dots, 1, a_{i_1}^{-m_{i_1}}\cdots a_{i_k}^{-m_{i_k}}, 1, \dots, 1)_1\]
and 
\[(1, \dots, 1, g, 1, \dots, 1, a_{i_1}^{-m_{i_1}}\cdots a_{i_k}^{-m_{i_k}}, 1, \dots, 1)_1 \equiv (1, \dots, 1, g, 1, \dots, 1, g^{-1}, 1, \dots, 1)_1 \bmod G_n'.\]
\end{proof}

\begin{remark}
Proposition \ref{prop:elementinverse} is not true when $n=3$ which can be seen from the generators for $\Stab_{G_3}(1)$ obtained in \cite{Ski19}. This significantly contributes to the change in the rigid kernels for $G_n$ starting at $n=4$ described in Theorem \ref{thm:branchkernel}.
\end{remark}

\begin{corollary} 
\label{cor:reorderreplace}
For $n\geq 4$,  $(g_1, \dots, g_n)_1$ is in $\Stab_{G_n}(1)$ if and only if $v*(g_{1^\theta}\cdots g_{n^\theta})$ is in $\rist_{G_n}(1)$ for every vertex $v$ on the first level and every permutation $\theta$ of $\{1, \dots, n\}$.
\end{corollary}

\subsection{Rigid kernels}

Since by Lemma \ref{lem:altsym},  
\[ \langle a_1(\varnothing), a_2(\varnothing), \dots, a_n(\varnothing) \rangle =
\begin{cases} 
      S_n & \text{if } n \text{ is odd} \\
      A_n & \text{if } n \text{ is even}
   \end{cases}
\]
and since the normal subgroup structure of the alternating and symmetric groups changes starting at $n=5$, we will split the next computations into 3 settings, when $n=4$, when $n\geq 5$ is odd, and when $n\geq 5$ is even.

\begin{proposition}
\label{prop:stab4.1}
$\Stab_{G_4}(1)= \langle a_1a_3a_4^2, a_2a_1a_3a_1^{-1}, a_1a_3^{-1}a_4a_3, X*G_4', I_4\rangle$.
\end{proposition}

\begin{proof}
This is done by using the Reidemeister-Schreier method for finding generators of a subgroup or by using the GAP \cite{GAP4} package AutomGrp \cite{Atm} and eliminating redundant generators.
\end{proof}

Put $K_4=\langle a_1a_3a_4^2, a_2a_1a_3a_1^{-1}, a_1a_3^{-1}a_4a_3, G_4'\rangle$ and note that it has index $3$ in $G_4$ by Proposition \ref{prop:abel}.

\begin{theorem}
\label{thm:4max}
The subgroup $K_4$ is normally generated in $G_4$ by the set $\{ a_1a_2, a_2a_3, a_3a_4, a_4a_1 \}$ and is the maximal branching subgroup for $G_4$. In particular, for all $m$, the rigid stabilizer of the $m$-th level is precisely $X^m*K_4$ and the stabilizer of the $m$-th level is $X^{m-1}*\Stab_{G_4}(1)$. Consequently, $\Stab_{G_4}(m+1)\leq \rist_{G_4}(m)$ for all $m$ and $G_4$ has a trivial rigid kernel.
\end{theorem}

\begin{proof}
Since $G_4'\leq K_4$, $K_4$ is a normal subgroup. Moreover, $ a_2a_1a_3a_1^{-1}\equiv a_2a_3 \bmod G_n'$ and similarly $a_1a_3^{-1}a_4a_3\equiv a_4a_1 \bmod G_n'$. Further, $a_3a_4$ and $a_1a_2$ can be written as a product of the generators of $K_4$ and their conjugates. Let $\widetilde{K_4}=\langle \langle  a_1a_2, a_2a_3, a_3a_4, a_4a_1 \rangle \rangle$. Now clearly $\widetilde{K_4} \leq K_4$. Since $G_4/ K_4\cong \mathbb{Z}/3 \mathbb{Z}$ to check that $\widetilde{K_4}=K_4$ it suffices to show that $G_4/\widetilde{K_4}$ has order at most three. This is immediate from the fact that 
\[a_1 \equiv a_2^{-1} \equiv a_3 \equiv a_4^{-1} \bmod \widetilde{K_4}\]
and that each $a_i$ has order 3. 

Now to show that $K_4$ is a branching subgroup, by self-similarity it is only necessary to show that $K_4 \geq X*K_4$. Consider the elements
\[a_1 a_3 a_4^2=(a_1, a_3, 1, a_4^2)_1,\]
\[a_2a_1a_3a_1^{-1}=(a_1^{-1}, a_2a_3, 1,a_1)_1,\]
\[a_1a_3^{-1}a_4a_3=(a_1a_4, a_3^{-1},a_3,1)_1.\]

Applying Corollary \ref{cor:reorderreplace}, $(a_1 a_3 a_4^2,1,1,1)_1$, $(a_2a_1a_3a_1^{-1},1,1,1)_1$, and $(a_1a_3^{-1}a_4a_3,1,1,1)_1$ are in $K_4$ (since only elements in the commutator subgroup are required to shift the coordinates).

To show that $K_4$ is the maximal branching subgroup, observe that the generators obtained in Proposition \ref{prop:stab4.1} for $\Stab_{G_4}(1)$ generate a subgroup of index 3 in $X*G_4$ and so in particular, for each vertex $v$ on the first level $\rist_{G_4}(v)$ must be a proper subgroup of $v*G_4$ and thus any branching subgroup must also be a proper subgroup of $G_4$. Since $K_4$ has index 3 and $K_4$ contains $v*K_4$ for every vertex $v$, $K_4$ is the maximal branching subgroup and $X*K_4=\rist_{G_4}(1)$.

By Theorem \ref{thm:branchingcommutator}, Proposition \ref{prop:elementinverse}, and Proposition \ref{prop:stab4.1}, $\Stab_{G_4}(1) \leq K_4$ and the rest follows from self-similarity. 
\end{proof}

We note for the reader that the group $G_4$ was studied first by Siegenthaler in Chapter 6 of \cite{sie09} providing theorems which seem to be in opposition to Theorem \ref{thm:4max}. Indeed, Siegenthaler correctly notes that often in the calculations for regular branch groups one can replace the maximal branching subgroup with any branching subgroup containing the normal core of the maximal one. He then proceeds to incorrectly work with an arbitrary normal branching subgroup which is most evident when his results from this chapter are applied to the group $G_4$. With the exception of Theorem 6.2.3., the proposition, theorems, and corollaries in this chapter remain true if one replaces the hypothesis that $K$ is an arbitrary normal branching subgroup with the hypothesis that $K$ is the normal core of the maximal branching subgroup.

Now we move to odd $n \geq 5$. Recall that for $g\in G_n$ we define $\epsilon(g)=(\sum_{j=1}^k r_i) \bmod (n-1)$ where $w(s_1, \dots, s_n)=s_{i_1}^{r_1}s_{i_2}^{r_2}\cdots s_{i_k}^{r_k}$ is a word with $w(a_1, \dots, a_n)=g$. Recall also that if $g=(g_1, \dots, g_n)\sigma$, then $\epsilon_1(g)=\sum_{i=1}^n \epsilon(g_i) \bmod{(n-1)}$ and by Lemma \ref{lem:epsilonequiv}, $\epsilon(g)=\epsilon_1(g)$.

\begin{proposition}
\label{prop:oddstab1}
For odd $n \geq 5$, if $g\in G_n$ is in $\Stab_{G_n}(1)$, then $\epsilon(g)\equiv 0 \bmod 2$. Moreover, if $g_1, \dots, g_n$ are arbitrary elements of $G_n$ with $\sum_{i=1}^n\epsilon(g_i)\equiv 0 \bmod 2$, then there exists $g\in G_n$ with $g=(g_1, \dots, g_n)_1$ in $\Stab_{G_n}(1)$.
\end{proposition}

\begin{proof}
First observe that since the $a_i(\varnothing)$ is an element in $S_n \backslash A_n$ for all $i$, if a word in $a_1, \dots, a_n$ produces an element $g$ in $\Stab_{G_n}(1)$, then it necessarily has even exponent sum. In particular, $\epsilon(g)\equiv 0 \bmod 2$.

Recall that $I_n$ is the set of all elements of the form $(1, \dots, 1, g, 1, \dots, 1, g^{-1},1,\dots,1)_1$ and that $I_n \subseteq \Stab_{G_n}(1)$. Define $H_n=\langle I_n, X*G_n' \rangle \unlhd G_n.$ Observe that $G_n/ X*G_n'$ is isomorphic to a subgroup of 
\[(\mathbb{Z}/(n-1)\mathbb{Z})^n \wr S_n=\big((\mathbb{Z}/(n-1)\mathbb{Z})^n \times \cdots \times (\mathbb{Z}/(n-1)\mathbb{Z})^n\big) \rtimes S_n.\]
and hence $G_n/H_n$ isomorphic to a subgroup of $(\mathbb{Z}/(n-1)\mathbb{Z})^n \times S_n$. We claim that in fact $G_n/H_n$ is a subdirect product of $(\mathbb{Z}/(n-1)\mathbb{Z})^n \times S_n$. Indeed, $H_n$ is contained in the kernel of $\epsilon_1$, a surjective homomorphism onto $(\mathbb{Z}/(n-1)\mathbb{Z})^n$, and $G_n$ surjects onto $S_n$.

Let $\pi_1: G_n/H_n \twoheadrightarrow S_n$ and let $\pi_2: G_n/H_n \twoheadrightarrow (\mathbb{Z}/(n-1)\mathbb{Z})^n$.  Identify the kernel of $\pi_1$ with a subgroup of $(\mathbb{Z}/(n-1)\mathbb{Z})^n$ and the kernel of $\pi_2$ with a subgroup of $S_n$. By Goursat's Lemma, $(\mathbb{Z}/(n-1)\mathbb{Z})^n/ \ker(\pi_1) \cong S_n/\ker(\pi_2)$. Since the only non-trivial abelian quotient of $S_n$ has order $2$, $(\mathbb{Z}/(n-1)\mathbb{Z})^n/ \ker(\pi_1)$ is either trivial or order 2. But since a word in $a_1, \dots, a_n$ has a trivial permutation only if it has even exponent sum, $\ker(\pi_1)$ is a proper subgroup of $(\mathbb{Z}/(n-1)\mathbb{Z})^n$. Therefore, \[\Stab_{G_n}(1)=\{ (g_1, \dots, g_n)_1 \mid \sum_{i=1}^n \epsilon(g_i) \equiv 0 \bmod 2\}.\]
\end{proof}

\begin{definition}
For odd $n$, define $K_n=\{g \in G_n \mid \epsilon(g) \equiv 0 \bmod 2\} \leq G_n$.
\end{definition}

\begin{theorem}
\label{thm:oddrist}
For odd $n \geq 5$, $K_n$ is the maximal branching subgroup for $G_n$. Moreover, $\rist_{G_n}(m)=X^m*K_n$ for all $m$. Consequently, 
\[\Stab_{G_n}(m+1)=X^{m}*\Stab_{G_n}(1) \leq \rist_{G_n}(m)\]
for all $m$ and $G_n$ has a trivial rigid kernel.
\end{theorem}

\begin{proof}
Again, it suffices to show that $K_n\geq X*K_n =\rist_{G_n}(1)$.
By Proposition \ref{prop:oddstab1}, $(1,\dots, 1,g,1, \dots, 1)_1\in G_n$ if and only if $\epsilon(g)\equiv 0 \bmod 2$ which is if and only if $g\in K_n$. Moreover, by Lemma \ref{lem:epsilonequiv} such an $(1,\dots, 1, g, 1, \dots, 1)_1$ is in $K_n$. By Lemma \ref{lem:epsilonequiv} and Proposition \ref{prop:oddstab1}, $\Stab_{G_n}(1)\leq K_n$ and the rest follows from the above work.
\end{proof}

Now, we work with the remaining groups: $G_n$ where $n\geq 5$ is even.

\begin{definition}
A group $G \leq \Aut( \mathcal{T})$ is called layered if $G$ contains the direct product of $|X|$ copies of $G$ each acting on one of the subtrees of $\mathcal{T}$ rooted at the first level, i.e.
\[X*G\leq G.\]
\end{definition}

\begin{theorem}
\label{thm:evensubgroups}
For even $n\geq 5$, $\Stab_{G_n}(m)=\rist_{G_n}(m)=X^m*G_n$. In particular, $G_n$ is layered and consequently $G_n$ has a trivial rigid kernel.
\end{theorem}

\begin{proof}
It suffices to show for $m=1$.
Let $H_n$ be as in the proof of Proposition \ref{prop:oddstab1}. By the same arguments presented there, for even $n\geq 5$, $G_n/H_n$ isomorphic to a subgroup of $(\mathbb{Z}/(n-1)\mathbb{Z})^n \times A_n$ (since the root permutations generate $A_n$ by Lemma \ref{lem:altsym}). This time, $G_n/H_n$ is a subdirect product of $(\mathbb{Z}/(n-1)\mathbb{Z})^n \times A_n$ as $H_n$ is again contained in the kernel of $\epsilon_1$ and $G_n$ surjects onto $A_n$.

Let $\pi_1: G_n/H_n \twoheadrightarrow A_n$ and let $\pi_2: G_n/H_n \twoheadrightarrow (\mathbb{Z}/(n-1)\mathbb{Z})^n$. By Goursat's Lemma, $(\mathbb{Z}/(n-1)\mathbb{Z})^n/ \ker(\pi_1) \cong A_n/\ker(\pi_2)$. Since the only abelian quotient of $A_n$ is the trivial group, $(\mathbb{Z}/(n-1)\mathbb{Z})^n/ \ker(\pi_1)$ is trivial and so $\Stab_{G_n}(1)=X*G_n$. Since $X*G_n$ is in fact a direct product, it is also $\rist_{G_n}(1)$. Moreover, as $X*G_n\leq G_n$, for any $m$ we have $X^m*G_n\leq G_n$. Since the group is self-similar, the result follows.
\end{proof}

Note that Theorem \ref{thm:evensubgroups} tells us that for even $n \geq 5$, $G_n=G_n \wr A_n$. In particular, this implies the following corollary.

\begin{corollary}
For even $n\geq 5$, $\overline{G_n}=(\cdots A_n \wr A_n)\wr A_n)\wr \cdots A_n)$, the infinitely iterated wreath product of $A_n$.
\end{corollary}

\subsection{Branch kernels}

The combination of Theorems \ref{thm:4max}, \ref{thm:oddrist}, and \ref{thm:evensubgroups} show that unlike when $n=3$, when $n\geq 4$ the congruence kernel for $G_n$ is the same as the branch kernel. The following is extracted from the proof of Theorem 2.7 in \cite{BSZ12}.

\begin{theorem*}[\cite{BSZ12}]
Let $G$ be a branch group. Then the branch kernel is 
\[\varprojlim_{\substack{e\geq 1 \\ m\geq 1}} \rist_G(m)/\rist_G(m)'\rist_G(m)^e\]
\end{theorem*}

\begin{theorem}
\label{thm:branchkernel}
For $n \neq 3$, the branch kernel, and thus the congruence kernel, for $G_n$ is the inverse limit
\[\varprojlim_{m\geq 1} M_n^m\]
 where $M_n$ is a finite abelian group. When $n\geq 5$ is even, $M_n$ is cyclic of order $n-1$ and when $n=4$ or $n\geq 5$ is odd, $M_n$ has exponent bounded between $(n-1)$ and $2(n-1)$.
\end{theorem}

\begin{proof}
For $n=4$,  $\rist_{G_4}(m)/\rist_{G_4}(m)' \cong (K_4)^{4^m}/(K_4')^{4^m}=(K_4/K_4')^{4^m}$. Now $K_4$ is a subgroup of index $3$ containing $G_4'$ and hence surjects onto a subgroup of index $3$ in $G_4/G_4'=(\mathbb{Z}/3\mathbb{Z})^4$. The image of $K_4$ is then an abelian group of exponent $3$ and so $K_4/K_4'$ has exponent at least $3$. It is easy to check that the normal generators of $K_4$ given by Theorem \ref{thm:4max} have order 6. Since conjugating does not change the order of an element, $K_4$ has a generating set consisting of elements of order 6 and so $K_4/K_4'$ has exponent at most 6. Now since $K_4$ has finite index in a finitely generated group, it is finitely generated. Therefore $K_4/K_4'$ is a finite abelian group with exponent between 3 and 6.

For odd $n \geq 5$, $\rist_{G_n}(m)/\rist_{G_n}(m)' \cong (K_n)^{n^m}/(K_n')^{n^m}=(K_n/K_n')^{n^m}$. Now $K_n$ is a subgroup of index $2$ containing $G_n'$ and as such surjects onto a subgroup of index $2$ in $G_n/G_n'=(\mathbb{Z}/(n-1)\mathbb{Z})^n$. Since $n\geq 5$, the image of $K_n$ is an abelian group of exponent $(n-1)$ and so $K_n/K_n'$ has exponent at least $(n-1)$. Moreover, since $n$ is odd, a generating set for $K_n$ is $\{a_{n-1}a_1, a_na_2, a_ia_{i+2} \mid 1 \leq i \leq n-2 \}$. It is easy to check that each of these elements has order $2(n-1)$. Thus $K_n/K_n'$ has exponent at most $2(n-1)$. 

For even $n\geq 5$, 
\[\rist_{G_n}(m)/\rist_{G_n}(m)' \cong (G_n)^{n^m}/(G_n')^{n^m}=(G_n/G_n')^{n^m}=\big((\mathbb{Z}/ (n-1)\mathbb{Z})^n\big)^{n^m}.\]

Now since for all $n\geq 4$, $\rist_{G_n}(m)/\rist_{G_n}(m)'$ has finite exponent, the collection $\{\rist_{G_n}(m)/\rist_{G_n}(m)'\}$ is cofinal with $\{\rist_{G_n}(m)/\rist_{G_n}(m)'\rist_{G_n}(m)^e\}$. Further, since 
 
\[
  \rist_{G_n}(m)/\rist_{G_n}(m+1) =
  \begin{cases}
       (G_n/G_n')^{n^m} & \text{if $n \geq 5$ is even} \\
       (K_n/K_n')^{n^m} & \text{if $n=4$ or $n\geq 5$ is odd} 
 \end{cases}
\]

we see that similarly $\{ (G_n/G_n')^m \}$ and $\{(K_n/K_n')^m \}$ respectively also form cofinal sets with $\rist_{G_n}(m)/\rist_{G_n}(m+1)$. In particular, the branch kernel is \[\varprojlim_{m \geq 1} M_n^m\]

where
\[
 M_n =
  \begin{cases}
       G_n/G_n' & \text{if $n\geq 5$ is even} \\
       K_n/K_n' & \text{if $n=4$ or $n\geq 5$ is odd} 
 \end{cases}
\]
\end{proof}

\begin{remark}
Our techniques only put bounds on the exponent of $K_n/K_n'$ for $n=4$ and odd $n\geq 5$. It would be desirable to precisely understand this group.
\end{remark}

\subsection{Just Infinite-ness}

\begin{definition}
A group is said to be just infinite if it is infinite but every proper quotient is finite.
\end{definition}

In \cite{Gri00} Theorem 4, a criterion for determining when a branch group is just infinite is posed. 

\begin{theorem*}[\cite{Gri00}]
A branch group $G$ is just infinite if and only if for each $m\geq 1$, the index of $\rist_G(m)'$ in $\rist_G(m)$ is finite.
\end{theorem*}

In \cite{BSZ12}, it is shown the $G_3/G_3''$ is an infinite group and so $G_3$ is not just infinite. For $n\geq 4$, the proof of Theorem \ref{thm:branchkernel} shows $\rist_{G_n}(m)/\rist_{G_n}(m)'$ is finite. Thus we obtain the following result.

\begin{theorem}
\label{thm:justinfinite}
$G_n$ is just infinite if and only if $n\neq 3$.
\end{theorem}

\section{Maximal subgroups}\label{sec:finiteindex}
In this final section, we present examples to show that triviality of rigid kernel is not necessarily preserved when moving to subgroups of finite index, even if they are maximal. In doing so, we present new examples of branch groups with non-trivial rigid kernels, adding to the only currently known example of the Hanoi towers group.

\begin{theorem}
\label{thm:finiteindex}
For $n\geq 4$, let $d >2$ be such that $d \mid (n-1)$ and let $H_{n,d}$ be the set of elements $g$ of $G_n$ with $\epsilon(g)\equiv 0 \bmod d$. The $H_{n,d}$ is a subgroup of index $d$ in $G_n$ and is a regular branch group with non-trivial rigid kernel.
\end{theorem}
\begin{proof}
First observe that $H_{n,d}$ contains $G_n'$ and thus also contains $X^m*G_n'$ for any $m$. Since $G_n'$ acts as $A_n$ on the top level, $H_{n,d}$ is level transitive. Further, since $X^m*G_n'$ has finite index in $G_n$ it also has finite index in $H_{n,d}$ and we conclude that $H_{n,d}$ is a regular branch group.

For any $n$ and $d$ as in the theorem, $H_{n,d}$ has index $d$ in $G_n$ by Lemma \ref{lem:surject}. We will construct explicit elements that are in $\Stab_{H_{n,d}}(m)$ but not in $\rist_{H_{n,d}}(k)$ for all $k \leq m$. Let $\beta=a_1a_2\cdots a_n$.

If $n\geq 4$ is odd, then 
\[\beta=(a_1a_3\cdots a_n, 1, \dots, 1, a_2a_4\cdots a_{n-1})(1,n)\]
and 
\[\beta^2=(a_1a_3\cdots a_na_2a_4\cdots a_{n-1},1,\dots, 1, a_2a_4\cdots a_{n-1}a_1a_3\cdots a_n)_1.\] 
Clearly, $\beta^2$ has exponent sum $2n$ and is not an element of $H_{n,d}$. But $H_{n,d}$ does contains $G_n'$ and therefore also $X*G_n'$ and all elements of the form $(g,1, \dots, 1, g^{-1})_1$ for $g\in G_n$. Combining these elements we get that
$\beta^2 \equiv (\beta^2, 1, \dots ,1)_1 \bmod H_{n,d}$ and so likewise $(\beta^2, 1, \dots, 1)_1$ is not contained in $H_{n,d}$. Inductively we get for any $m$, 
\[\beta^2 \equiv (\beta^2, 1, \dots, 1)_m \bmod H_{n,d}\]
and so $(\beta^2, 1, \dots, 1)_m$ is not contained in $H_{n,d}$.

But again, since $H_{n,d}$ contains all elements of the form $(g, 1, \dots, 1, g^{-1})_1$, the element $(\beta^2, 1, \dots, 1, \beta^{-2})_1 \in \Stab_{H_{n,d}}(1)$ and again inductively for all $m$,  $(\beta^2, 1, \dots, 1, \beta^{-2})_m \in \Stab_{H_{n,d}}(m)$. But $(\beta^2, 1, \dots, 1, \beta^{-2})_m \notin \rist_{H_{n,d}}(k)$ for any $k$, otherwise  $(\beta^2, 1, \dots, 1)_{m-k}$ would be in the group $H_{n,d}$, a contradiction.

Now if $n\geq 4$ is even, then
\[\beta=(a_1a_3\cdots a_{n-1}, 1, \dots, 1, a_2a_4\cdots a_n)_1\]
and so by the same discussion above $\beta \notin H_{n,d}$ and for all $m$ 
\[\beta \equiv (\beta, 1, \dots, 1)_m \bmod H_{n,d}\]
so $(\beta, 1, \dots, 1)_m$ is not an element of $H_{n,d}$ but $(\beta, 1, \dots, 1, \beta^{-1})_m$ is. The same arguments show that $(\beta, 1, \dots, 1, \beta^{-1})_m$ is not in $\rist_{H_{n,d}}(k)$ for any $k$.
\end{proof}

\section{Hausdorff Dimension}\label{sec:haus}
For a closed subgroup $H$ of $\Aut(\mathcal{T})$, the Hausdorff dimension of $H$ can be calculated \cite{bs97} by
\begin{equation}
\label{eq:haus}
dim_{\mathrm{H}}(H)=\liminf_{m\rightarrow \infty} \frac{\log|H/\Stab_H(m)|}{\log |\Aut(\mathcal{T})/\Stab_{\Aut(\mathcal{T})}(m)|}. 
\end{equation}
Ab\'ert and Vir\'ag showed that with probability 1 the closure of the subgroup generated by three random automorphisms of a binary tree has Hausdorff dimension 1 \cite{av05}. Siegenthaler then constructed the first explicit examples of topologically finitely generated groups of Hausdorff dimension 1 \cite{sie08}.

As a consequence of the work in previous sections, we show that $\overline{G_n}$ has Hausdorff dimension arbitrarily close to 1. 
\begin{theorem}
\label{thm:haus}
For $n\geq 3$, the Hausdorff dimension for $\overline{G_n}$ is
\[
  dim_{\mathrm{H}}(\overline{G_n}) =
  \begin{cases}
  	   \vspace{.2cm}
       1-\frac{\log(48)}{\log(331776)} & \text{if $n=4$} \\
       \vspace{.2cm}
       1-\frac{\log(2)}{\log(n!)} & \text{if $n\geq 5$ is even} \\
       1-\frac{\log(2)}{n\log(n!)} & \text{if $n$ is odd}
  \end{cases}
\]
\end{theorem}

\begin{proof}
For $n=4$, $|G_4/\Stab_{G_4}(1)|=|A_4|=\frac{4!}{2}$ by Lemma \ref{lem:altsym}. It can easily be checked from the generators of Proposition \ref{prop:stab4.1} that $\Stab_{G_4}(1)/\Stab_{G_4}(2)$ is an index 3 subgroup of $A_4\times A_4 \times A_4 \times A_4$, so $|\Stab_{G_4}(1)/\Stab_{G_4}(2)|=\frac{4!^4}{3\cdot 2^4}$. For $m\geq 2$, $|\Stab_{G_4}(m-1)/\Stab_{G_4}(m)|=|\Stab_{G_4}(1)/\Stab_{G_4}(2)|^{4^{m-2}}$ by Theorem \ref{thm:4max}. Hence equation \ref{eq:haus} yields

\begin{align*}
dim_{\mathrm{H}}(\overline{G_4})  &=\liminf_{m\rightarrow \infty} \frac{\log\big(\frac{4!^{1+4+\cdots + 4^{m-1}}}{ 2^{1+4+\cdots 4^{m-1}}3^{1+4+\cdots 4^{m-2}}}\big)}{\log(4!^{1+4+\cdots 4^{m-1}})}\\
 &=\liminf_{m\rightarrow \infty} \frac{\log(4!^{\frac{4^m-1}{3}})-\log(2^{\frac{4^m-1}{3}})-\log(3^{\frac{4^{m-1}-1}{3}})}{\log(4!^{\frac{4^m-1}{3}})}\\
 &=\liminf_{m\rightarrow\infty} 1-\frac{\log(2)}{\log(4!)}-\frac{(4^{m-1}-1)\log(3)}{(4^m-1)\log(4!)}\\
 &=1-\frac{\log(2)}{\log(4!)}-\frac{\log(3)}{4\log(4!)}\\
 &=1-\frac{\log(48)}{\log(331776)}.
\end{align*}

For even $n\geq 5$, $G_n/\Stab_{G_n}(1)=A_n$ and $\Stab_{G_n}(m-1)/\Stab_{G_n}(m)=(A_n)^{n^{m-1}}$ by Lemma \ref{lem:altsym} and Theorem \ref{thm:evensubgroups}. Therefore

\begin{align*}
dim_{\mathrm{H}}(\overline{G_n})  &= \frac{\log\big(\frac{n!^{1+n+\cdots n^{m-1}}}{2^{1+n+\cdots n^{m-1}}}\big)}{\log(n!^{1+n+\cdots n^{m-1}})}\\
    &=\liminf_{m\rightarrow \infty} \frac{\log((n!)^{\frac{n^m-1}{n-1}})-\log(2^{\frac{n^m-1}{n-1}})}{\log(n!^{\frac{n^m-1}{n-1}})}\\
    &=1-\frac{\log(2)}{\log(n!)}.
\end{align*}

Finally, when $n$ is odd $|G_n/\Stab_{G_n}(1)|=|S_n|=n!$ by Lemma \ref{lem:altsym}. Additionally, $|\Stab_{G_n}(1)/\Stab_{G_n}(2)|=\frac{n!^n}{2}$ by Proposition \ref{prop:oddstab1} and Lemma 5.7 in \cite{Ski19}. Moreover, $|\Stab_{G_n}(m-1)/\Stab_{G_n}(m)|=|\Stab_{G_n}(1)/\Stab_{G_n}(2)|^{4^{m-2}}=\frac{n!^{n^{m-1}}}{2^{n^{m-2}}}$ by Theorem \ref{thm:evensubgroups} and Lemma 5.8 in \cite{Ski19}. Thus

\begin{align*}
dim_{\mathrm{H}}(\overline{G_n})  &=\liminf_{m \rightarrow \infty} \frac{\log\big(\frac{n!^{1+n+\cdots n^{m-1}}}{2^{1+n+\cdots n^{m-2}}}\big)}{\log(n!^{1+n+\cdots n^{m-1}})}\\
   &=\liminf_{m\rightarrow \infty} \frac{\log(n!^{\frac{n^m-1}{n-1}})-\log(2^{\frac{n^{m-1}-1}{n-1}})}{\log(n!^{\frac{n^m-1}{n-1}})}\\
   &=\liminf_{m\rightarrow \infty} 1- \frac{(n^{m-1}-1)\log(2)}{(n^m-1)\log(n!)}\\
   &=1-\frac{\log(2)}{n\log(n)}.
\end{align*}

\end{proof}

\begin{corollary}
For all $\epsilon >0$, there exists $n$ such that $dim_{\mathrm{H}}(\overline{G_n})>1-\epsilon$.
\end{corollary}

\bibliographystyle{alpha}

\end{document}